\documentclass[12pt,amssymb]{amsart}

\usepackage{graphicx}  
\usepackage[abs]{overpic}

\usepackage[usenames]{color}
\usepackage{txfonts}
\usepackage[dvipdfm,bookmarks=true,bookmarksnumbered=true,bookmarkstype=toc,colorlinks=true,linkcolor=blue,citecolor=OliveGreen]{hyperref}

\newtheorem{thm}{Theorem}
\newtheorem{lem}[thm]{Lemma}

\theoremstyle{definition}
\newtheorem{defn}[thm]{Definition}

\theoremstyle{remark}
\newtheorem{rem}[thm]{Remark}

\newcommand{\refeq}[1]{\textup{(\ref{eq:#1})}}
\newcommand{\refthm}[1]{Theorem \ref{thm:#1}}
\newcommand{\reflem}[1]{Lemma \ref{lem:#1}}

\newcommand{\refdefn}[1]{Definition \ref{defn:#1}}
\newcommand{\refrem}[1]{Remark \ref{rem:#1}}

\numberwithin{equation}{section}
\numberwithin{thm}{section}

\newcommand{\refi}[1]{\textup{(\ref{i:#1})}}

\newcommand{\const}{C}

\newcounter{const}
\newcommand{\cnst}{\refstepcounter{const}\const_{\theconst}}
\newcommand{\clabel}[1]{\cnst\label{#1}}
\newcommand{\refc}[1]{\const_{\ref{c:#1}}} 

\newcommand{\del}{\delta}
\newcounter{del}
\newcommand{\delcnst}{\refstepcounter{del}\del_{\thedel}}
\newcommand{\dlabel}[1]{\delcnst\label{#1}}
\newcommand{\refd}[1]{\del_{\ref{d:#1}}}

\newcommand{\qtext}{\quad\text}

\newcommand{\sm}{\setminus}

\newcommand{\grad}{\nabla}

\newcommand{\ve}{\varepsilon}
\newcommand{\vphi}{\varphi}

\newcommand{\bd}{\partial}         
\newcommand{\cl}[1]{\overline{#1}} 

\newcommand{\ol}[1]{\overline{#1}} 
\newcommand{\til}[1]{\tilde{#1}} 
\newcommand{\wtil}[1]{\widetilde{#1}} 
\newcommand{\str}[1]{{#1}^*} 

\newcommand{\R}{{\mathbb R}}
\newcommand{\C}{{\mathbb C}}

\newcommand{\D}{D}      
\newcommand{\U}{U}      

\DeclareMathOperator{\dist}{dist}
\DeclareMathOperator{\diam}{diam}

\newcommand{\inv}{^{-1}}

\newcommand{\normi}[1]{\|#1\|_\infty} 
\newcommand{\normL}[1]{\|#1\|_{{\rm{Lip}}}} 

\makeatletter

\newcommand{\pder}[2]{\if@display\dfrac{\partial#1}{\partial#2}
\else\partial#1/\partial#2\fi}

\newcommand{\union}{\mathop{
\if@display
\bigcup\else\operatorname{\hbox{\small$\bigcup$}}
\fi}}

\makeatother

\newcommand{\Dom}{D}
\newcommand{\Bdy}{\bd\dom}
\newcommand{\dom}{\Omega}
\newcommand{\bdy}{\bd\Omega}		    

\begin{document}

\title[The growth of the vorticity gradient]
{The growth of the vorticity gradient for the two-dimensional Euler
flows on domains with corners}

\author{Tsubasa Itoh}
\address{
Department of Mathematics,
Tokyo Institute of Technology,
Oh-okayama Meguro-ku Tokyo 152-8551, Japan
}
\email{tsubasa@math.titech.ac.jp}

\author{Hideyuki Miura}
\address{
Graduate School of Information Science 
and Engineering Mathematical and Computing Sciences, 
Tokyo Institute of Technology,
Oh-okayama Meguro-ku Tokyo 152-8551, Japan
}
\email{miura@is.titech.ac.jp}

\author{Tsuyoshi Yoneda}
\address{
Department of Mathematics,
Tokyo Institute of Technology,
Oh-okayama Meguro-ku Tokyo 152-8551, Japan
}
\email{yoneda@math.titech.ac.jp}


\begin{abstract}
We consider the two-dimensional Euler  equations 
in non-smooth domains with corners. 
It is shown that  if the angle of the corner $\theta$ is strictly less than $\pi/2$, 
the Lipschitz estimate of the vorticity at the corner is 
at most single exponential growth and the upper bound is sharp.
For the corner with the larger angle $\pi/2 < \theta <2\pi$, 
$\theta \neq \pi$, 
we construct an example of the vorticity which 
loses continuity instantaneously. 
For the case $\theta \le \pi/2$, 
the vorticity remains continuous inside the domain.
We thus identify the threshold of the angle for the vorticity maintaining the continuity.
For the borderline angle $\theta=\pi/2$, it is also shown that the growth rate of 
the Lipschitz constant of the vorticity 
can be double exponential, which is the same as in Kiselev-Sverak's result 
(Annals of Math., 2014).
\end{abstract}

\maketitle





\section{Introduction}

Let $\dom$ be a two-dimensional domain. We 
are concerned with the Euler equations in $\Omega$ in the vorticity formulation:
\begin{equation}\label{eq:Euler}
\omega_t+(u\cdot\grad)\omega=0,
\quad \omega(x,0)=\omega_0(x).
\end{equation}
Here $\omega$ is the fluid vorticity, and 
$u$ is the velocity 
of the flow determined by the Biot-Savart law. 
We impose the no flow condition for the velocity at the boundary: 
$u \cdot n = 0$ on $\bdy$, 
where $n$ is the unit normal vector on the boundary.
This implies the formula:
\begin{equation}\label{eq:u}
u(x,t)=\grad^{\bot}\int_{\dom}G_{\dom}(x,y)\omega(y,t)dy,
\end{equation}
where $G_{\dom}$ is the Green function for the Dirichlet problem in $\dom$ 
and $\grad^{\bot}=(\partial_{x_2},-\partial_{x_1})$.
The movement of a fluid particle, 
placed at a point $X\in \dom$, 
is defined as the solution of the Cauchy problem
\begin{equation}\label{eq:trajectory}
\frac{d\gamma_X(t)}{dt}=u(\gamma_X(t),t),\quad\gamma_X(0)=X,
\end{equation}
and the vorticity $\omega$ is advected by
\begin{equation}\label{eq:vorticity}
\omega(x,t)=\omega_0(\gamma_x^{-1}(t)).
\end{equation}

Global regular solutions to the Euler equatins \eqref{eq:Euler} in smooth bounded domains 
were proved by Wolibner \cite{W} and H\"{o}lder \cite{H}
and there are huge literature on this problem.
Recently, there are growing interests in the study of \eqref{eq:Euler}  
in nonsmooth domains. Existence of  global weak solutions, 
with $u\in L^{\infty}(\R_+;L^{2}(\dom))$ 
and $\omega\in L^{\infty}(\R_+ \times \dom)$, 
was proved by Taylor \cite{Tay} for convex domains 
and by G\`erard-Varet and Lacave \cite{GL}
for more general (possibly not convex) domains.
Uniqueness of the solution to the Euler equations \refeq{Euler} on domains 
with corners was shown
by Lacave, Miot and Wang \cite{LMW} for acute angles.
For obtuse corners, 
Lacave \cite{L} proved uniqueness of the solution 
under the assumption that  the support of
the vorticity never intersects the boundary.
We are concerned with the question 
how fast the maximum of the gradient of the vorticity can grow 
as $t \rightarrow \infty$. 
When $\dom$ is a smooth bounded domain, 
the best known upper bound on the growth is double-exponential \cite{Y},
while the question whether such upper bound is sharp 
had been open for a long time. 
In 2014, 
Kiselev and Sverak \cite{KS} answered the question affirmatively
for the case $\dom$ is a disk. 
They gave an example of the solution growing with double exponential rate.
For a general domain with $C^3$-boundary see \cite{Xu}.
On the other hand, Kiselev and Zlatos \cite{KZ} considered
the 2D Euler flows on some bounded domain with certain cusps.
They showed that the gradient of vorticity blows up at the cusps 
in finite time. 
These solutions are constructed by imposing certain symmetries on 
the
initial data, which leads 
to a \textit{hyperbolic flow scenario} 
near a stagnation point on the boundary. 
More precisely, 
by the hyperbolic flow scenario,
particles on the boundary (near the stagnation point) head for the stagnation point for all time.
Moreover the relation between this scenario and 
the geometry of the boundary plays a crucial role in 
the double exponential growth or the formation of the singularity.  
Thus it would be an interesting question to ask how the geometry of the 
boundary affects the growth of the solution. 
In \cite{IMY} the authors considered 
the Euler equations \refeq{Euler} on the unit square
and under a simple symmetry condition 
the growth of 
the Lipschitz constant of
the vorticity
on the boundary is shown to be 
at most single exponential at the stagnation point.
In this paper, we are concerned with more general cases;
the growth of the Lipschitz norm of the vorticity
in  bounded domains with  general corners.

\begin{defn}\label{defn:corner}~(i)~
Let $\dom\subset\R^{2}$ be a simply connected bounded domain $0<\theta <2\pi$ with $\theta \neq \pi$.
We say that $\bdy$ has 
a corner of angle $\theta$ $(0<\theta<2\pi)$ at
$\xi\in\bdy$,
if there exist constants $r_0>0$ and $0\le\theta_0<2\pi$ such that,
$\dom\cap B(\xi,r_0)
=\{x=(x_1,x_2):\theta_0 <\arg(x-\xi)<\theta_0+\theta \}\cap B(\xi,r_0)$.
\\
(ii)~Let $\Omega$ be a domain with corners given in (i). We say $\Omega$ is symmetric 
with respect to the corner
if $\theta_0=-\frac{\theta}{2}$ and  $\Omega$ is symmetric along the $x_1$-axis.

\end{defn}

Without loss of generality, by translation, rotation and scaling, we may assume that
\begin{equation}\label{eq:domain}
\begin{cases}
&\text{$\diam(\dom)<1$ and $0\in \bdy$,}\\
&\text{$\bdy$ has a corner of angle $\theta$ at $0$ with $\theta_0=0$ in \refdefn{corner}}.
\end{cases}
\end{equation}

We now focus on the growth of the Lipschitz constant of $\omega$ with 
$a\in\cl\dom$
\[
\sup_{x\in\cl\dom}
\frac{|\omega(x,t)-\omega(a,t)|}{|x-a|}.
\]

Our first result concerns the domain with the corner with 
the angle $\theta \le \pi/2$. 

\begin{thm}\label{thm:acute}
Let  $\dom$ be a simple connected domain satisfying \refeq{domain} 
and $\omega_0$ be a Lipschitz function. 
\\
(a)~For $0<\theta<\frac{\pi}{2}$, 
there exists a constant $\const>0$ depending only on $\dom$
such that  
\begin{equation}\label{eq:uppersingle}
\sup_{x\in\dom}
\frac{|\omega(x,t)-\omega(0,t)|}{|x|}
\le \normL{\omega_0}e^{\const \normi{\omega_0}t}
\qtext{ for $t>0$}.
\end{equation}
Moreover there exist an initial data $\omega_0$ 
and a constant $\const>0$ such that
\begin{equation}\label{eq:lowersingle}
\sup_{x\in\dom}
\frac{|\omega(x,t)-\omega(0,t)|}{|x|}
\ge
\const e^{\const t}
\qtext{ for $t>0$}.
\end{equation}
(b)~For $\theta=\frac{\pi}{2}$, there exists an initial data $\omega_0$ with $\normL{\omega_0}>1$ such that
\begin{equation}\label{eq:lowerdouble}
\sup_{x\in\dom}
\frac{|\omega(x,t)-\omega(0,t)|}{|x|}
\ge
\normL{\omega_0}^{\const\exp(\const t)}
\qtext{ for $t>0$}.
\end{equation}
(c)~If $\bdy$ is $C^{1,1}$ except at $0\in\bdy$, 
then there exists a constant $\const$ depending only on $\dom$
such that  
\begin{equation}\label{eq:Holder}
|\omega(x,t)-\omega(y,t)|\le \normL{\omega_0} 
|x-y|^{\exp(-\const\normi{\omega_0}t)}
\qtext{ for $x,y\in\dom$ and $t>0$}.
\end{equation}
\end{thm}
\begin{rem}
The assertion (a) shows that if $\theta<\frac{\pi}{2}$,
then the growth of the Lipschitz constant at the corner 
of the vorticity
is at most single exponential 
and the upper bound is sharp. 
For the case $\theta=\frac{\pi}{2}$, one can see from (b) that 
there exists an initial data $\omega_0$ such that
the growth of the Lipschitz constant of the vorticity 
at the corner is at least double-exponential.
In our argument, we are imposing an infimum condition to the initial vorticity: $\inf_{x\in\dom}\omega_0>0$ (see \reflem{ulower}).
This condition makes the proof simpler.
Indeed, we do not need a bootstrapping argument as 
in the proof of \cite[Theorem 1.1]{KS} anymore.
The assertion (c) shows that the vorticity remains continuous in $\Omega$
although the H\"older exponent is decreasing in $t$.
It is likely that the solution is Lipschitz continuous in $\Omega$  
and the growth is at most exponential. We would like to address this issue elsewhere.
\end{rem}
\begin{rem}
For the case $\theta =\frac{\pi}{2}$, we could not figure out whether the upper bound is indeed double-exponential. 
In fact we are analyzing local behavior of the flow near the corner 
by using the conformal mapping 
and the Green function of the unit upper half-disk. 
For smooth domains, 
it follows $C^{1,\alpha}$-regularity of the velocity on $\cl \dom$, 
we can obtain the double exponential upper bound 
without using conformal mappings. 
See \cite[Theorem 2.1 and Proposition 2.2]{KS} for example.
\end{rem}

\begin{rem}
Assume $\Omega$ is symmetric with respect to the corner and 
$\omega_0(x_1,x_2)=-\omega_0(x_1,-x_2)$ in $x\in \Omega$.
Then, by Theorem \ref{thm:acute}, we can immediately see that  if $\theta \in (0,\pi)$, then its corresponding solution has also single exponential bound. In this point of view, 
Theorem \ref{thm:acute} can be considered as a generalization of \cite{IMY}.
To obtain the upper bound, we split the domain $\Omega$ into
$\Omega \cap \{x:x_2>0\}$, and just apply Theorem \ref{thm:acute} to the splitted domain  (with the half angle $\theta/2$ case). 
In this case we do not need the infimum condition $\inf_{x\in\Omega} \omega_0(x)>0$ (see Lemma \ref{lem:ulower}) anymore.
\end{rem}

We next consider the case $\theta >\pi/2$. 
In this case, we will see that the vorticity can lose continuity 
instantaneously.
\begin{thm}\label{thm:obtuse}
Let  $\dom$ be a simply connected bounded domain 
satisfying \refeq{domain}. If $\pi/2<\theta<\pi$, 
there are an initial data $\omega_0\in C(\cl\dom)$ 
and its solution $\omega$ 
such that $\omega(t)$ instantaneously loses continuity in space.
Furthermore, if $\pi<\theta<2\pi$ and  $\dom$ is symmetric with respect to the corner, 
there also exist $\omega_0\in C(\cl\dom)$ and its solution $\omega$ 
such that $\omega(t)$ instantaneously loses continuity.

\end{thm}

In the proof of our results, the estimates of the velocitiy fields near the corner 
 play important roles as in \cite{KS, Xu, IMY}. 
One of the new ingredients in our proof is to use of the conformal mapping
which have not used for the large time behavior of the vorticity.
This enables us to obtain the explicit representation of the Green function $G_{\dom}$ 
in the Biot-Savart law \label{eq:u} via the conformal mapping 
and to estimate the behavior of the velocity fields near the corner. 
Finally, we note that Theorems \ref{thm:acute} and 
\ref{thm:obtuse} hold for domains with more general corners or even finite number of corners;
see \refrem{corner}.

We use the following notation.
By the symbol $\const$ we denote an absolute positive constant
whose value is unimportant and may change from one occurrence to the next.
If necessary, we use $\const_0, \const_1, \dots$, to specify them.
We say that $f$ and $g$ are comparable and write $f\approx g$
if two positive quantities $f$ and $g$ satisfies $\const\inv \le f/g\le\const$
with some constant $\const\ge1$.
The constant $\const$ is referred to as the constant of comparison.
We have to pay attention for the dependency of
the constant of comparison.
For $x=(x_1,x_2)\in\R^2$ we let $\til x =(-x_1,x_2)$, 
$\ol{x} =(x_1,-x_2)$ and 
$\str{x} =x/|x|^2$. 
Let $m$ be the two-dimensional Lebesgue outer measure.

\section{Preliminaries}
Let $\Dom$ be a bounded simply connected open subset of $\R^2$. 
Identifying $\R^2$ with $\C$, 
the Riemann mapping theorem states that 
there exists a conformal mapping $f$ of the open unit disk $\D= B(0,1)$ 
onto $\Dom$. 
Moreover Carath{\'e}odory theorem asserts that
if $\Dom$ is Jordan domain, 
then $f$ has a continuous injective extension to $\ol{\D}$.
If $\Dom$ is  $C^{1,1}$-domain,
then the following Kellogg-Warschawski theorem holds.
See \cite[Theorem 3.6]{P} and 
\cite[Theorem I\hspace{-.1em}I.4.3 and Lemma I\hspace{-.1em}I.4.4]{GM}. 

\begin{thm}\label{KW}
Let $f$ be a conformal map 
from $\D$ onto a $C^{1,1}$-domain $\Dom$.
Then $f'$ has a continuous extension to $\cl \D$ and
\begin{equation}
\frac{f(\zeta)-f(z)}{\zeta-z}\to f'(z)\neq 0 
\qtext{for }\,\, \zeta\to z,\,\,\,\zeta,z\in\cl\D,
\end{equation}
\begin{equation}
|f'(z_1)-f'(z_2)|\le\const|z_1-z_2|\log|z_1-z_2|^{-1}
\qtext{for }z_1,z_2\in\cl\D,\quad|z_1-z_2|<1.
\end{equation}
\end{thm}

The following theorem states the smoothness of conformal map $f:\D\to\Dom$ 
in a neighborhood of $f^{-1}(\zeta)$ depends only 
on the smoothness of $\Bdy$ in a neighborhood of $\zeta\in\Bdy$.
See \cite[Theorem I\hspace{-.1em}I.4.1]{GM}.

\begin{thm}\label{thm:localsmooth}
Let $\Dom_1$ and $\Dom_2$ be Jordan domains such that
$\Dom_1\subset\Dom_2$ and 
let $\gamma\subset\Bdy_1\cap\Bdy_2$ be an open subarc.
Let $\vphi_j$ be a conformal map of $\D$ onto $\Dom_j$ $(j=1,2)$.
Then $\psi=\vphi_2^{-1}\circ\vphi_1$ has an analytic continuation 
across $\vphi_1^{-1}(\gamma)$, and $\psi'\neq 0$ on $\vphi_1^{-1}(\gamma)$.
\end{thm}

Let $\U=\{z=x+iy\in\D:y>0\}$.
Using the above theorems, we show the following lemma.

\begin{lem}\label{lem:conformal}
Assume that $\dom$ satisfies \refeq{domain}. 
Let $\beta=\pi/\theta$.
Then there exists a conformal map $f:\dom \to \U$ with $f(0)=0$.
Let $g\equiv f^{-1}$.
Moreover there exist a constant $\dlabel{d:conformal}>0$ 
such that
\begin{enumerate}
\item\label{i:fcompara} 
$|f(z)|\approx|z|^{\beta}$
and $|f'(z)|\approx|z|^{\beta-1}$ 
for $z\in\ol{\dom}\cap B(0,\refd{conformal})$,
\item\label{i:gcompara} 
$|g(w)|\approx|w|^{\frac{1}{\beta}}$
and $|g'(w)|\approx|w|^{\frac{1}{\beta}-1}$
for $w\in \ol{\U} \cap B(0,\refd{conformal})$.
\end{enumerate}
Here $\refd{conformal}$ and 
the constant $\clabel{c:conformal}$
of comparison 
depend only on $\dom$.
\end{lem}

\begin{proof}
Let $\vphi(z)=z^{\beta}$.
Observe that $\vphi(0)=0$ and 
$\partial\vphi(\dom)$ is locally a straight line near 0.
By the Riemann mapping theorem, 
there exists a conformal map $f_1:\vphi(\dom)\to\U$ with $f_1(0)=0$.
Let $g_1=f_1^{-1}$.
The Kellogg-Warschawski theorem and 
\refthm{localsmooth} imply that
there is a constant $\delta>0$ such that
\[
|f'_1(z)|\approx 1,\quad |g'_1(w)|\approx 1,\quad
|g_1(w)|\approx |w|
\]
for 
$z\in \ol{\vphi(\dom)}\cap B(0,\delta)$ and
$w\in \ol{\U}\cap B(0,\delta)$.
Let $f=f_1\circ \vphi$ and 
let $\refd{conformal}<\delta$ be a sufficiently small constant.
Then we see that $f(0)=0$.
Since $f'(z)=f'_1(\vphi(z))\vphi'(z)$ and
$g'(w)=\vphi^{-1}(g_1(w))g'_1(w)$, 
we have
\begin{equation}\label{eq:f'compara}
|f'(z)|\approx |\vphi'(z)|\approx |z|^{\beta-1}
\end{equation}
and
\begin{equation}\label{eq:g'compara}
|g'(w)|\approx |(\vphi^{-1})(g_1(w))|
\approx |g_1(w)|^{\frac{1}{\beta}-1}
\approx |w|^{\frac{1}{\beta}-1}
\end{equation}
for 
$z\in \ol{\vphi(\dom)}\cap B(0,\refd{conformal})$ and
$w\in \ol{\U}\cap B(0,\refd{conformal})$.
It follows from the above estimates and the mean-value property 
with $f(0)=0$ and $g(0)=0$ that
\[
|f(z)| \approx |z|^{\beta},\quad 
|g(w)| \approx |w|^{\frac{1}{\beta}}
\]
for 
$z\in \ol{\vphi(\dom)}\cap B(0,\refd{conformal})$ and
$w\in \ol{\U}\cap B(0,\refd{conformal})$.
Thus the properties \refi{fcompara},\refi{gcompara} hold.
\end{proof}

\begin{rem}\label{rem:corner}
Alternatively, we claim that Theorems \ref{thm:acute} and 
\ref{thm:obtuse} hold for domains with 
a more general corners.
Let $\gamma(s)$ be a parametrization of $\bdy$ with $\gamma(0)=0$.
We consider a domain such that
$\gamma$ is $C^{1,1}$-Jordan curve except at $0\in\bdy$ and
$\lim_{s \searrow 0}\arg{\gamma(s)-\gamma(-s)}=\theta$.
In the proof of \reflem{conformal}, 
we would also need a condition
\begin{equation}\label{eq:gcorner}
\text{$(\gamma(s))^{\beta}$ is $C^{1,1}$ close to $0\in\bdy$},
\end{equation}
in order to use the Kellogg-Warschawski theorem. 
Then \reflem{conformal} holds for domains with a general corner.
For simplicity, we assume that $\dom$ satisfies \refeq{domain}. 
\end{rem}

\section{The key lemmas}
To prove Theorems \ref{thm:acute} and \ref{thm:obtuse}, 
we need a technical lemma for the expansion of velocity field.
Assume that $\dom$ satisfies \refeq{domain}.
Since the Green function for the unit upper half-disk 
$\U$ is given explicitly by 
\[
G_{\U}(x,y)=\frac{1}{2\pi}(\log|x-y|-\log|x-\str{y}|
                             -\log|\ol{x}-y|+\log|\ol{x}-\str{y}|),
\]
the Green function for $\dom$ is given explicitly by
\[
\begin{split}
G_{\dom}(x,y)&=G_{\U}(f(x),f(y)) \\
&=\frac{1}{2\pi}\big(\log|f(x)-f(y)|-\log|f(x)-\str{f(y)}|
                 -\log|\ol{f(x)}-f(y)|+\log|\ol{f(x)}-\str{f(y)}|\big),
\end{split}
\]
where $f$ is the conformal map of $\dom$ onto $\U$ 
in \reflem{conformal}.
Let 
\[
\begin{split}
G(x,y)=\log
\frac{\big|f(x)-f(y)\big|}
{\big|\ol{f(x)}-f(y)\big|},\qquad
G^*(x,y)=\log
\frac{\big|\ol{f(x)}-\str{f(y)}\big|}
{\big|f(x)-\str{f(y)}\big|}
\end{split}
\]
for $x,y\in\dom$.
Firstly we get an upper bound of $u$ near the corner.

\begin{lem}\label{lem:uupper}
Let $0<\theta<\pi$ and $\beta=\pi/\theta$.
Assume that $\dom$ satisfies \refeq{domain}.
There exists a constant $\const>0$ depending only on 
$\dom$ such that 
\begin{equation}\label{eq:uupper}
|u(x,t)|
\le \const\normi{\omega_0}
\begin{cases}
|x| 
&\qtext{if $\beta>2$}, \\
|x|\log|x|^{-1}
&\qtext{if $\beta=2$}, \\
|x|^{\beta-1}
&\qtext{if $1<\beta<2$},
\end{cases}
\end{equation}
for $x\in\dom$ and $t>0$.
In particular, we see that $u(0,t)=0$ for any $t>0$.
\end{lem}

\begin{proof}
Let $f,g,\refc{conformal}$ and $\refd{conformal}$ be as 
in \reflem{conformal}.
Let $\delta$ be a small positive constant to be determined later and 
let $x\in\dom$.
It is sufficient to show that \refeq{uupper} for $|x|<\delta$.
By \reflem{conformal}\refi{fcompara}, 
we observe that
there exist $0<\ve<\refd{conformal}$ such that
if $y\in\dom_+$ and $|y|\ge\ve$,
then $|f(y)|\ge\Big(\frac{\ve}{\refc{conformal}}\Big)^{\beta}$.
Let us to be $\delta<\frac{\ve}{(2\refc{conformal}^2)^{\frac{1}{\beta}}}$.
Then we have
\begin{equation}\label{eq:u1sep}
{\footnotesize
\begin{split}
u_1(x,t)
&=
\frac{1}{2\pi}\int_{\dom\cap B(0,\ve)}
\pder{G}{x_2}(x,y)\omega(y,t)dy 
+\frac{1}{2\pi}\int_{\dom\sm B(0,\ve)}
\pder{G}{x_2}(x,y)\omega(y,t)dy 
+\frac{1}{2\pi}\int_{\dom}
\pder{G^{*}}{x_2}(x,y)\omega(y,t)dy.
\end{split}
}
\end{equation}
for $|x|<\delta$.

Firstly we estimate 
the last term of the right hand side of \refeq{u1sep}.
Assume that $\delta<\frac{1}{(2\refc{conformal})^{\beta}}$.
Since $|\str{f(y)}|\ge 1$,
we have $|f(x)|\le\refc{conformal}|x|^{\beta}\le\frac{1}{2}
\le\frac{1}{2}|\str{f(y)}|$, 
so that
\begin{equation}\label{eq:*estimate}
{\footnotesize
\big|\ol{f(x)}-\str{f(y)}\big|
\ge\big|f(x)-\str{f(y)}\big| 
\ge \frac{1}{2}|\str{f(y)}|
\ge \frac{1}{2}.
}
\end{equation}
We have
\[
{\footnotesize
\begin{split}
\pder{G^{*}}{x_2}(x,y)
&=
-\frac{(f_1(x)-\str{f_1}(y))\pder{f_1}{x_2}(x)
       +(f_2(x)-\str{f_2}(y))\pder{f_2}{x_2}(x)}
      {\big|f(x)-\str{f(y)}\big|^2}
+\frac{(f_1(x)-\str{f_1}(y))\pder{f_1}{x_2}(x)
       +(f_2(x)+\str{f_2}(y))\pder{f_2}{x_2}(x)}
      {\big|\ol{f(x)}-\str{f(y)}\big|^2}, \\
\end{split}
}
\]
where $f(x)=(f_1(x),f_2(x))$ and $\str{f(y)}=(\str{f_1}(y),\str{f_2}(y))$.
Thus for $y\in\dom_+$, we have
\[
\Bigg|\pder{G^{*}}{x_2}(x,y)\Bigg|
\le\const\frac{|f'(x)|}{\big|f(x)-\str{f(y)}\big|}
\le\const|x|^{\beta-1},
\]
by \refeq{*estimate} and \reflem{conformal}\refi{fcompara}.
Therefore we have
\begin{equation}\label{eq:G*}
\Bigg|\int_{\dom}
\pder{G^{*}}{x_2}(x,y)\omega(y,t)dy\Bigg|
\le\const
|x|^{\beta-1}\normi{\omega_0}.
\end{equation}

Next we estimate 
the second term of the right hand side of \refeq{u1sep}.
Let $y\in\dom\sm B(0,\ve)$.
Assume that $\delta<\frac{\ve}{(2\refc{conformal})^{\frac{1}{\beta}}
\refc{conformal}}$.
Then $|f(x)|\le\refc{conformal}|x|^{\beta}\le\frac{1}{2}\Big(\frac{\ve}{\refc{conformal}}\Big)^{\beta}
\le\frac{1}{2}|f(y)|$, 
so that 
\[
\big|\ol{f(x)}-f(y)\big|
\ge \big|f(x)-f(y)\big|
\ge \frac{1}{2}|f(y)|
\ge \frac{1}{4}\Big(\frac{\ve}{\refc{conformal}}\Big)^{\beta}. 
\] 
Since
\[
{\footnotesize
\begin{split}
\pder{G}{x_2}(x,y)&=
+\frac{\big(f_1(x)-f_1(y)\big)\pder{f_1}{x_2}(x)
       +\big(f_2(x)-f_2(y)\big)\pder{f_2}{x_2}(x)}
     {\big|f(x)-f(y)\big|^{2}}
-\frac{\big(f_1(x)-f_1(y)\big)\pder{f_1}{x_2}(x)
       +\big(f_2(x)+f_2(y)\big)\pder{f_2}{x_2}(x)}
      {\big|\ol{f(x)}-f(y)\big|^{2}}, \\
\end{split}
}
\]
we have
\[
\Bigg|\pder{G}{x_2}(x,y)\Bigg|
\le\const\frac{|f'(x)|}{\big|f(x)-f(y)\big|}
\le\const|x|^{\beta-1},
\]
by \reflem{conformal}\refi{fcompara}.
Therefore we have
\begin{equation}\label{eq:Gout}
\Bigg|\int_{\dom\sm B(0,\ve)}
\pder{G}{x_2}(x,y)\omega(y,t)dy\Bigg|
\le\const
|x|^{\beta-1}\normi{\omega_0}.
\end{equation}

Finally we consider 
the first term of the right hand side of \refeq{u1sep}.
In this case, the singularity at $x=y$ appears.
So we need to calculate more carefully.
We have
\[
\Bigg|\int_{\dom\cap B(0,\ve)}
\pder{G}{x_2}(x,y)\omega(y,t)dy\Bigg|
\le\const
|x|^{\beta-1}\normi{\omega_0}
\int_{\dom\cap B(0,\ve)}
\frac{dy}{\big|f(x)-f(y)\big|}
\]
by \reflem{conformal}\refi{fcompara}.
Let $z=f(x)$ and $\ve'=\refc{conformal}\ve^{\beta}$.
The substitution $w=f(y)$ yields
\[
{\footnotesize
\begin{split}
\int_{\dom\cap B(0,\ve)}
\frac{dy}{\big|f(x)-f(y)\big|}
&\le
\int_{\U\cap B(0,\ve')}
\frac{dw}{\big|z-w\big|\big| f'(y) \big|^{2}} \\
&\le\const
\int_{B(0,\ve')}
\frac{dw}{\big|z-w\big|\big| w \big|^{2-2/\beta}} \\
&\le\const
\Bigg(\int_{B(0,\frac{1}{2}|f(x)|)}
\frac{dw}{\big|z-w\big|\big| w \big|^{2-2/\beta}}
+
\int_{B(z,\frac{1}{2}|f(x)|))}
\frac{dw}{\big|z-w\big|\big| w \big|^{2-2/\beta}} \\
&\qquad\qquad\qquad\qquad\qquad
+\int_{B(0,\ve')\sm
[B(0,\frac{1}{2}|f(x)|)
\cup B(z,\frac{1}{2}|f(x)|)]}
\frac{dw}{\big|z-w\big|\big| w \big|^{2-2/\beta}}
\Bigg) \\
&\le\const
\Bigg(
|x|^{-\beta}
\int_{0}^{\frac{1}{2}|f(x)|}
r^{\frac{2}{\beta}-1}dr
+
|x|^{2-2\beta}
\int_{0}^{\frac{1}{2}|f(x)|}
dr 
+\int_{\frac{1}{2}|f(x)|}^{\ve'}r^{\frac{2}{\beta}-2} dr
\Bigg) \\
&\le
\begin{cases}
\const|x|^{2-\beta} 
&\qtext{if $\beta>2$}, \\
\const\log|x|^{-1}
&\qtext{if $\beta=2$}, \\
\const
&\qtext{if $1<\beta<2$},
\end{cases}
\end{split}
}
\]
by \reflem{conformal}\refi{fcompara}.
Therefore we have
\begin{equation}\label{eq:Gin}
\Bigg|\int_{\dom\cap B(0,\ve)}
\pder{G}{x_2}(x,y)\omega(y,t)dy\Bigg|
\le \const\normi{\omega_0}
\begin{cases}
|x| 
&\qtext{if $\beta>2$}, \\
|x|\log|x|^{-1}
&\qtext{if $\beta=2$}, \\
|x|^{\beta-1}
&\qtext{if $1<\beta<2$}.
\end{cases}
\end{equation}

Combining \refeq{G*},\refeq{Gout} and \refeq{Gin}, 
we have 
\[
|u_1(x,t)|
\le \const\normi{\omega_0}
\begin{cases}
|x| 
&\qtext{if $\beta>2$}, \\
|x|\log|x|^{-1}
&\qtext{if $\beta=2$}, \\
|x|^{\beta-1}
&\qtext{if $1<\beta<2$}.
\end{cases}
\]
In a similar way to the proof of this estimate, 
we obtain that 
\[
|u_2(x,t)|
\le \const\normi{\omega_0}
\begin{cases}
|x| 
&\qtext{if $\beta>2$}, \\
|x|\log|x|^{-1}
&\qtext{if $\beta=2$}, \\
|x|^{\beta-1}
&\qtext{if $1<\beta<2$},
\end{cases}
\]
so that \refeq{uupper} holds.
\end{proof}

If $\theta\le\frac{\pi}{2}$ and $\bdy$ is $C^{1,1}$ except at $0\in\bdy$, 
then the velocity $u$ is log-Lipchitz continuous on $\cl\dom$.
In a way similar to the proof of \cite[Proposition 3.4]{LMW} 
we obtain the following lemma.

\begin{lem}\label{lem:logLip}
Let $0<\theta\le\frac{\pi}{2}$.
Assume that $\dom$ satisfies \refeq{domain} 
and $\dom$ is $C^{1,1}$ except at $0\in\bdy$.
There exists a constant $\const>0$ depending only on 
$\dom$ such that 
\begin{equation}\label{eq:logLip}
|u(x,t)-u(y,t)|
\le \const\normi{\omega_0}
|x-y|\log|x-y|^{-1}
\end{equation}
for $x,y\in\dom$ and $t>0$.
\end{lem}

Next we get a lower bound of $u$ near the corner.

\begin{lem}\label{lem:ulower}
Assume that $\dom$ satisfies \refeq{domain}. 
Let $\beta=\pi/\theta$. \\
(a) Let $0<\theta\le\frac{\pi}{2}$.
If $c_0=\min_{x\in\dom}\omega_0>0$,
then
there exist constants $\dlabel{d:ulower1}>0$ 
and $\clabel{c:ulower1}>0$ depending only on 
$\dom,\normi{\omega_0}$ and $c_0$ such that 
\begin{equation}\label{eq:ulower1}
u_1(x,t)\le-\refc{ulower1} 
\begin{cases}
x_1 
&\qtext{if $0<\theta<\frac{\pi}{2}$}, \\
x_1\log x_1^{-1}
&\qtext{if $\theta=\frac{\pi}{2}$}, 
\end{cases}
\end{equation}
for $x=(x_1,0)\in\bdy,0<x_1<\refd{ulower1}$ and $t>0$. \\
(b)
Let $\frac{\pi}{2}<\theta<\pi$.
If $\omega_0>0$
then
there exist constants $\dlabel{d:ulower2}>0$ 
and $\clabel{c:ulower2}>0$ depending only on 
$\dom$ and $\normi{\omega_0}$ such that 
\begin{equation}\label{eq:ulower2}
u_1(x,t)\le-\refc{ulower2}
x_1^{\beta-1}
\end{equation}
for $x=(x_1,0)\in\bdy,0<x_1<\refd{ulower2}$ and $t>0$. 
\end{lem}

\begin{proof}
Let $f$ be as 
in \reflem{conformal}.
Let $\delta$ be a small positive constant.
Now we consider the particle behavior on the boundary.
Let $x=(x_1,0)\in\bdy,0<x_1<\delta$.
Observe that $f_2(x)=0$, $\pder{f_2}{x_1}(x)=0$ and 
$\pder{f_2}{x_2}(x)\ge\const|x|^{\beta-1}\ge\const x_1^{\beta-1}$. 

(a)
Assume that $0<\theta\le\frac{\pi}{2}$ and 
$c_0=\min_{x\in\dom}\omega_0>0$.
Since \refeq{vorticity}, we see that
$\min_{y\in\dom}\omega(y,t)=c_0$ for any $t>0$.
Let $\ve$
be as 
in the proof of \reflem{uupper}.
By \refeq{G*} and \refeq{Gout},
we have
\[
\begin{split}
u_1(x,t)
&\le-\frac{1}{\pi}\pder{f_2}{x_2}(x)\int_{\dom\cap B(0,\ve)}
\frac{f_2(y)}{|f(x)-f(y)|^2}\omega(y,t)dy
+\const\normi{\omega_0}|x|^{\beta-1} \\
&\le-\const x_1^{\beta-1}\int_{\dom\cap B(0,\ve)}
\frac{f_2(y)}{|f(x)-f(y)|^2}dy
+\const\normi{\omega_0}\const x_1^{\beta-1}.
\end{split}
\]
Let $z=f(x)$ and $\ve'=\refc{conformal}^{-1}\ve^{\beta}$.
The substitution $w=f(y)$ yields
\[
\begin{split}
\int_{\dom\cap B(0,\ve)}
\frac{f_2(y)}{|f(x)-f(y)|^2}dy
&\ge
\const
\int_{\U\cap B(0,\ve')}
\frac{w_2}{|z-w|^2|w|^{2-2/\beta}}dw \\
&\ge 
\const
\int_{2|z|}^{\ve'}\int_{0}^{\pi}
\frac{\sin\theta}{r^{2-2/\beta}}d\theta dr \\
&\ge 
\const
\int_{2|z|}^{\ve'}
r^{-2+2/\beta} dr \\
&\ge 
\const
\begin{cases}
x_1^{-\beta+2} 
&\qtext{if $\beta>2$}, \\
\log x_1^{-1}
&\qtext{if $\beta=2$}, \\
\end{cases}
\end{split}
\]
by \reflem{conformal}\refi{fcompara}.
Therefore we obtain
\[
u_1(x,t)
\le-\const
\begin{cases}
x_1(1-\const x_1^{\beta-2})
&\qtext{if $\beta>2$}, \\
x_1(\log x_1^{-1} -\const)
&\qtext{if $\beta=2$}. \\
\end{cases}
\]
We can choose $\delta>0$ sufficiently small so that 
\refeq{ulower1} holds. 

(b)
Assume $\frac{\pi}{2}<\theta<\pi$ and $\omega_0>0$ on $\dom$.
Since \refeq{vorticity}, we see that
$\omega(y,t)>0$ for any $y\in\dom$ and $t>0$.
Note that $\ol{f(x)}=f(x),f_2(x)=0$ for $x_2=0$
and $\str{f_1}(y)f_2(y)=f_1(y)\str{f_2}(y),f_2(y)=\str{f_2}(y)|f(y)|^2,
\str{f_2}(y)=f_2(y)|\str{f(y)}|^2$.
We obtain that
\[
\begin{split}
u_1(x)
&=\frac{1}{\pi}\pder{f_2}{x_2}(x)\int_{\dom}
\bigg(
-\frac{f_2(y)}{|f(x)-f(y)|^2}+\frac{\str{f_2}(y)}{|f(x)-\str{f(y)}|^2}
\bigg)\omega(y,t)dy \\
&=\frac{1}{\pi}\pder{f_2}{x_2}(x)\int_{\dom}
\frac{-f_1(x)^2(\str{f_2}(y)-f_2(y))-f_2(y)|\str{f(y)}|^2+\str{f_2}(y)|f(y)|^2}
{|f(x)-f(y)|^2|f(x)-\str{f(y)}|^2}
\omega(y,t)dy \\
&=\frac{1}{\pi}\pder{f_2}{x_2}(x)\int_{\dom}
\frac{(\str{f_2}(y)-f_2(y))(f_1(x)^2-1)}{|f(x)-f(y)|^2|f(x)-\str{f(y)}|^2}
\omega(y,t)dy \\
&\le
-\const x_1^{\beta-1}
\int_{\dom}
\frac{\str{f_2}(y)-f_2(y)}{|f(x)-\str{f(y)}|^2}
\omega(y,t)dy,
\end{split}
\]
where $f_1(x)^2-1\le -\const,\str{f_2}(y)-f_2(y)>0,|f(x)-f(y)|\le\const$ and 
$\omega(y,t)>0$ are used in the last inequality.
For sufficiently small $r>0$, we let
\[
\begin{split}
&\dom(r)=\{y\in\dom:\dist(y,\bdy)<r\}.
\end{split}
\]
If $y\in\dom\sm\dom(r)$, 
then
there exists a constant $\const$ such that
\[
\str{f_2}(y)-f_2(y)\ge \frac{1}{\const}
\text{\qquad and \qquad}
|f(x)-\str{f(y)}|\le\const,
\]
since $f(y)$ is away from the origin.
Then we have
\[
\int_{\dom}
\frac{\str{f_2}(y)-f_2(y)}{|f(x)-\str{f(y)}|^2}\omega(y,t)dy
\ge
\int_{\dom\sm\dom(r)}
\frac{\str{f_2}(y)-f_2(y)}{|f(x)-\str{f(y)}|^2}\omega(y,t)dy
\ge
\int_{\dom\sm\dom(r)}
\omega(y,t)dy
\ge\const,
\]
so that
\[
u_1(x)
\le
-\const x_1^{\beta-1}.
\]
Thus \refeq{ulower2} holds.
\end{proof}

\section{Proof of Theorems \ref{thm:acute} and \ref{thm:obtuse}}
In this section we will show 
Theorems \ref{thm:acute} and \ref{thm:obtuse}.

\begin{proof}[Proof of \refthm{acute}]
Let us consider the trajectory 
$\gamma_X(t)=(\gamma_{X}^1(t),\gamma_{X}^2(t))$ 
starting from a point $X\in\cl\dom$.
Let $x=\gamma_X(t)$. 

(a) Let $0<\theta<\frac{\pi}{2}$.
Assume that $\omega_0$ is Lipschitz.
By \reflem{uupper} and \refeq{trajectory},
we have
\[
\Bigg|\frac{d}{dt}|\gamma_{X}(t)|\Bigg|
\le\Bigg|\frac{d\gamma_{X}(t)}{dt}\Bigg|
\le 
\const\normi{\omega_0}|\gamma_{X}(t)|
\qtext{ for all }t>0,
\]
and so

\[
\frac{d}{dt}|\gamma_{X}(t)|
\ge -\const\normi{\omega_0}|\gamma_{X}(t)|
\qtext{ for all }t>0.
\]
By Gronwall's lemma we have
$|\gamma_{X}(t)|\ge |X| e^{-\const\normi{\omega_0}t}$, 
so that $|\gamma_{x}^{\inv}(t)|\le |x| e^{\const\normi{\omega_0}t}$.
Then we see that $\gamma_{0}^{\inv}(t)=0$.
Since $\omega(x,t)=\omega_0(\gamma_x^{\inv}(t))$ by the 2D Euler flows in the Lagrangian form, and $\omega_0$ is Lipschitz,
we obtain
\[
\begin{split}
|\omega(x,t)-\omega(0,t)|
&=|\omega_0(\gamma_x^{\inv}(t))-\omega_0(\gamma_0^{\inv}(t))| \\
&=|\omega_0(\gamma_x^{\inv}(t))-\omega_0(0)| \\
&\le \|\omega_0\|_{\text{Lip}}|\gamma_x^{\inv}(t)| \\
&\le \|\omega_0\|_{\text{Lip}} |x| e^{\const\normi{\omega_0}t}.
\end{split}
\]
Thus \refeq{uppersingle} holds.

Next
we consider an initial data $\omega_0$ defined by
\[
\omega_0(x)=|x|+1.
\]
Let $\refd{ulower1}$ be as in \reflem{ulower}(a).
Due to the boundary condition on $u$, the trajectories 
which start at the boundary stay on the boundary for all times.
We consider
the trajectory starting from a point $X=(X_1,0)\in\bdy$ 
with $0<X_1<\refd{ulower1}$.
Note that $\gamma_{X}^2(t)\equiv 0$ 
for any $t>0$.
By \reflem{ulower} and \refeq{trajectory},
we have
\[
\frac{d\gamma_{X}^1(t)}{dt}
\le 
-\refc{ulower1} \gamma_{X}^1(t)
\qtext{ for all }t>0.
\]
By Gronwall's lemma we have
$\gamma_{X}^1(t)\le X_1 e^{\refc{ulower1}t}$.
We obtain that
\[
\begin{split}
\sup_{x\in\cl\dom,x\neq 0}
\frac{|\omega(\gamma_{X}(t),t)-\omega(0,t)|}{|x|}
&\ge
\frac{|\omega(\gamma_{X}(t),t)-\omega(0,t)|}{|\gamma_{X}(t)|} \\
&=\frac{|\omega_0(X)-\omega_0(0)|}{\gamma_{X}^{1}(t)} \\
&\ge
\frac{|\omega_0(X)-\omega_0(0)|}{X_1} e^{\refc{ulower1}t}
=e^{\refc{ulower1}t}.
\end{split}
\]
Thus \refeq{lowersingle} holds.

(b)
For any $\ve>0$ 
we consider an initial data $\omega_0$ defined by
\[
\omega_0(x)=\min\Bigg\{\frac{|x|}{\ve}+1,2\Bigg\}.
\]
We see that $\normL{\omega_0}=\ve^{-1}$.
Let $\refd{ulower1}$ be as in \reflem{ulower}.
Note that $\refd{ulower1}$ is independent of $\ve$, 
instead depending on $\max\omega_0$ and $\min\omega_0$.
Assume that $\ve<\refd{ulower1}$.
Due to the boundary condition on $u$, the trajectories 
which start at the boundary stay on the boundary for all times.
We consider
the trajectory starting from a point $X=(\ve,0)\in\bdy$.
Note that $\gamma_{X}^2(t)\equiv 0$ 
for any $t>0$.
By \reflem{ulower} and \refeq{trajectory},
we have
\[
\frac{d\gamma_{X}^1(t)}{dt}
\le 
\refc{ulower1} \gamma_{X}^1(t)\log \gamma_{X}^1(t)
\qtext{ for all }t>0.
\]
By Gronwall's lemma we have
$\gamma_{X}^1(t)\le \ve^{\exp(\refc{ulower1}t)}$.
We obtain that
\[
\begin{split}
\sup_{x\in\cl\dom,x\neq 0}
\frac{|\omega(x,t)-\omega(0,t)|}{|x|}
&\ge
\frac{|\omega(\gamma_{X}(t),t)-\omega(0,t)|}{|\gamma_{X}(t)|} \\
&=\frac{|\omega_0(X)-\omega_0(0)|}{\gamma_{X}^{1}(t)} \\
&\ge
\ve^{-\exp(\refc{ulower1}t)} 
=\normL{\omega_0}^{\exp(\refc{ulower1}t)}.
\end{split}
\]
Thus \refeq{lowerdouble} holds.

(c) Let $0<\theta\le\frac{\pi}{2}$.
Assume that $\dom$ is $C^{1,1}$ except at $0\in\bdy$
and $\omega_0$ is Lipschitz.
Let $\gamma_Y(t)$ 
starting from a point $Y\in\cl\dom$.
Let $y=\gamma_Y(t)$. 
By \reflem{logLip} and \refeq{trajectory},
we have
\[
\Bigg|\frac{d}{dt}|\gamma_{X}(t)-\gamma_{Y}(t)|\Bigg|
\le\Bigg|\frac{d}{dt}\big(\gamma_{X}(t)-\gamma_{Y}(t)\big)\Bigg|
\le 
\const\normi{\omega_0}|\gamma_{X}(t)-\gamma_{Y}(t)|
\log |\gamma_{X}(t)-\gamma_{Y}(t)|^{-1}
\]
for all $t>0$,
and so
\[
\frac{d}{dt}|\gamma_{X}(t)-\gamma_{Y}(t)|
\ge \const\normi{\omega_0}|\gamma_{X}(t)-\gamma_{Y}(t)|
\log |\gamma_{X}(t)-\gamma_{Y}(t)|
\qtext{ for all }t>0.
\]
By Gronwall's lemma we have
\[
|\gamma_{X}(t)-\gamma_{Y}(t)|\ge |X-Y|^{\exp(\const\normi{\omega_0}t)}.
\] 
Thus we obtain that
\[
|\gamma_{x}^{\inv}(t)-\gamma_{y}^{\inv}(t)|
\le |x-y|^{^{\exp(-\const\normi{\omega_0}t)}}.
\]
Then we see that $\gamma_{0}^{\inv}(t)=0$.
Since $\omega(x,t)=\omega_0(\gamma_x^{\inv}(t))$ by the 2D Euler flows in the Lagrangian form, and $\omega_0$ is Lipschitz,
we obtain
\[
\begin{split}
|\omega(x,t)-\omega(y,t)|
&=|\omega_0(\gamma_x^{\inv}(t))-\omega_0(\gamma_y^{\inv}(t))| \\
&=|\omega_0(\gamma_x^{\inv}(t))-\omega_0(\gamma_y^{\inv}(t))| \\
&\le \|\omega_0\|_{\text{Lip}}|\gamma_x^{\inv}(t)-\gamma_y^{\inv}(t)| \\
&\le \|\omega_0\|_{\text{Lip}}|x-y|^{^{\exp(-\const\normi{\omega_0}t)}}.
\end{split}
\]
Thus \refeq{logLip} holds.
\end{proof}

\begin{proof}[Proof of \refthm{obtuse}]
Firstly we assume that $\frac{\pi}{2}<\theta<\pi$.
Let us consider the trajectory $\gamma_X(t)=(\gamma_{X}^1(t),\gamma_{X}^2(t))$ 
starting from a point $X\in\dom$.
Let $x=\gamma_X(t)$.
Now we consider a continuous initial data $\omega_0$ defined by
\[
\omega_0(x)=|x|.
\]
Let $\refd{ulower2}$ be as in \reflem{ulower}.
Due to the boundary condition on $u$, the trajectories 
which start at the boundary stay on the boundary for all times.
We consider
the trajectory starting from a point $X=(X_1,0)\in\bdy$ 
with $0<X_1<\refd{ulower2}$.
By \reflem{ulower} and \refeq{trajectory},
we have
\[
\frac{d\gamma_{X}^1(t)}{dt}
\le 
-\refc{ulower2} \Big(\gamma_{X}^1(t)\Big)^{\beta-1}
\qtext{ for all }t>0.
\]
By Gronwall's lemma we have
\[
\gamma_{X}^1(t)\le\Big(X_1^{2-\beta}-(2-\beta)\refc{ulower2} t )\Big)^{1/(2-\beta)}.
\]
Hence there exists $T_X\le X_1^{2-\beta}/(2-\beta)\refc{ulower2}$ 
such that $\gamma_{X}(T_X)=0$.
Note that $T_{X}\to 0$ as $X_1\to 0$.
On the other hand, \reflem{uupper} implies that $u(0,t)=0$, 
so $\gamma_{0}(t)\equiv 0$ is one of solutions of \refeq{trajectory}.
It follows from \refeq{vorticity} that
\[
\begin{split}
&\omega(\gamma_{0}(T_X),T_X)=\omega_0(0)=0, \\
&\omega(\gamma_{X}(T_X),T_X)=\omega_0(X)=|X|\neq 0.
\end{split}
\]
Since $\gamma_{X}(T_X)=\gamma_{0}(T_X)=0$, 
we see that
$\omega(\cdot,t)$ loses continuity
at $t=T_X$.

Next we assume that $\pi<\theta<2\pi$ 
and $\dom$ is symmetric with respect to the corner.
Without loss of generality, by rotation, we may assume that
$\bdy$ has a corner of angle $\theta$ at $0$ 
with $\theta_0=(\pi-\theta)/2$ in \refdefn{corner}.
Note that $\dom$ is symmetric with respect to the $x_2$-axis.
Now we consider a continuous initial data $\omega_0$ defined by
\begin{equation}\label{eq:pacman}
\omega_0(x)=x_1.
\end{equation}
Let $\wtil{\dom}=\{x\in\dom:x_1>0\}$.
Note that $\wtil{\dom}$ has a corner 
of angle $\theta/2$.
Define the function $\wtil{\omega_0}$ on $\wtil{\dom}$ 
by $\wtil{\omega_0}=\omega_0|_{\wtil{\dom}}$.
In a way similar to the above argument, 
there is a solution $\wtil{\omega}$ 
to the Euler equations \refeq{Euler} on $\wtil{\dom}$ 
such that $\wtil{\omega}(t)$ instantaneously loses continuity in space.
Now we define the function $\omega$ on $\dom$ 
by $\omega(\til{x})=-\omega(x)$ for $x\in\dom$. 
Then $\omega$ is one of solutions 
to the Euler equations \refeq{Euler} in $\dom$ with the initial data 
\refeq{pacman} 
and $\omega(t)$ instantaneously loses continuity in space.
\end{proof}

\vspace{0.5cm}
\noindent
{\bf Acknowledgments.}\ 
H.M. was partially supported by  Grant-in-Aid for Young Scientists (A), No.25707005 and
T.Y. was partially supported by Grant-in-Aid for Young Scientists (B),\\
 No. 25870004,
Japan Society for the Promotion of Science.
I.T. and T.Y. were supported by the ``Program to Promote the Tenure Track System" of the Ministry of Education, Culture, Sports, Science and Technology.  
We thank Professor Yasunori Maekawa for valuable comments.
We thank Professor Hisashi Okamoto for letting us know the Kraichnan-Leith-Batchelor theory in turbulence analysis. The theory must be related to our mathematical research, and clarifying these relation is  our future work.


\providecommand{\bysame}{\leavevmode\hbox to3em{\hrulefill}\thinspace}
\providecommand{\MR}{\relax\ifhmode\unskip\space\fi MR }
\providecommand{\MRhref}[2]{%
  \href{http://www.ams.org/mathscinet-getitem?mr=#1}{#2}
}
\providecommand{\href}[2]{#2}

\end{document}